\documentclass[11pt]{amsart}
\usepackage{mathrsfs,latexsym,amsfonts,amssymb}
\usepackage{xcolor}
\newtheorem{theorem}{Theorem}[section]
\newtheorem{lemma}[theorem]{Lemma}
\newtheorem{corollary}[theorem]{Corollary}
\newtheorem{question}[theorem]{Question}
\theoremstyle{remark}

\theoremstyle{definition}
\newtheorem{definition}[theorem]{Definition}

\newtheorem{example}[theorem]{Example}

\begin{document}

\title[On $\mathcal{I}$-covering images of metric spaces]
{On $\mathcal{I}$-covering images of metric spaces}

\author{Xiangeng Zhou}
\address{Xiangeng Zhou:  Department of Mathematics, Ningde Normal University, Ningde, Fujian 352100,
  P.R. China}
\email{56667400@qq.com}

\author{Shou Lin*}\thanks{*The corresponding author.}
\address{(Shou Lin): Institute of Mathematics,
Ningde  Normal University, Ningde, Fujian 352100, P.R. China}
\email{shoulin60@163.com}

%insert keywords
%insert subject class
\thanks{The project is supported by the NSFC (No. 12171015) and NSF of Fujian Province, China (No. 2020J01428, 2020J05230).}
%\date{\today}
\begin{abstract}
Let $\mathcal{I}$ be an ideal on $\mathbb{N}$. A mapping $f:X\to Y$ is called an $\mathcal{I}$-covering mapping provided a sequence $\{y_{n}\}_{n\in\mathbb N}$ is $\mathcal{I}$-converging to a point $y$ in $Y$, there is a sequence $\{x_{n}\}_{n\in\mathbb N}$ converging to a point $x$ in $X$ such that $x\in f^{-1}(y)$ and each $x_n\in f^{-1}(y_n)$. In this paper we study the spaces with certain $\mathcal{I}$-$cs$-networks and investigate the characterization of the images of metric spaces under certain $\mathcal{I}$-covering mappings, which prompts us to discover $\mathcal{I}$-$csf$-networks. The following main results are obtained:

(1)\,  A space $X$ has an $\mathcal{I}$-$csf$-network if and only if $X$ is a continuous and $\mathcal{I}$-covering image of a metric space.

(2)\, A space $X$ is an $\mathcal{I}$-$csf$-countable space if and only if $X$ is a continuous $\mathcal{I}$-covering and boundary $s$-image of a metric space.

(3)\,  A space $X$ has a point-countable $\mathcal{I}$-$cs$-network if and only if $X$ is a continuous $\mathcal{I}$-covering and $s$-image of a metric space.
\end{abstract}

\maketitle
\section{Introduction}
We know that mappings are an important tool to study spaces, and they play a pivotal role in discussing various images of metric spaces \cite{Ar66,LsYzq16}. Sequence-covering mappings are a special kind of mappings \cite{Sf71}. For example, every metric space is preserved by a continuous, sequence-covering and closed mapping \cite[Corollary 3.5.12]{Ls15}. Through sequence-covering mappings, we can establish the relationship between convergence sequences in topological spaces, and further reveal some topological properties based on convergent sequences \cite{Ls15}.

Convergence of sequences in a topological space is a basic and important concept in mathematics \cite{E, LsYzq16}. In addition to the usual convergence of sequences, statistical convergence \cite{{ab1},{ab6},RePr17,TangL}, ideal convergence \cite{{ba5},{ba7},ZZ21,{zl2020}} and $G$-convergence \cite{{ca2011},{Lin2006},MuCa18} have attracted extensive attention. In particular, ideals are a very useful notion in topology, analysis and set theory, and have been studied for along time. In \cite{Lin2020,cbb12}, it was studied that certain topological spaces are defined by ideal convergence.

Through various mappings, we have obtained rich results of ideal convergence. We know that every topological space is a continuous and sequence-covering image of a metric space.

\begin{question}
What topological spaces are characterized by continuous and $\mathcal{I}$-covering images of metric spaces?
\end{question}

In this paper we study the spaces with certain $\mathcal{I}$-$cs$-networks. In particular, the study of Question 1.1  prompts us to introduce the spaces with $\mathcal{I}$-$csf$-networks and $\mathcal{I}$-$csf$-countable spaces, and establish the relationship between these spaces and the images of metric spaces under certain $\mathcal{I}$-covering mappings.
These studies deepen our understanding for ideal topological spaces and mappings, present a version using the notion of ideals and provide a new research path for revealing the mutual relationship of spaces and mappings.

\section{Preliminaries}
In this section, recall some basic concepts related to this paper. The readers may refer to \cite{E, LsYzq16} for notation and terminology not explicitly given here.

By $\mathbb{N}$ we denote the set of positive integers.
Let $\mathcal A$ be the family of all subsets of the set $\mathbb{N}$.
An ${\it ideal}$ $\mathcal{I} \subseteq \mathcal A$ is a hereditary family of subsets of $\mathbb{N}$ which is stable under finite unions, i.e., the following are satisfied: if $B\subset A\in\mathcal{I}$, then $B\in\mathcal{I}$; if $A, B\in\mathcal{I}$, then $A\cup B\in\mathcal{I}$.
An ideal $\mathcal{I}$ is said to be ${\it non}$-${\it trivial}$, if $\mathcal{I}\neq \emptyset$ and $\mathbb{N}\notin\mathcal{I}$.
A non-trivial ideal $\mathcal{I}$ is called ${\it admissible}$ if $\mathcal{I}\supseteq \{\{n\}:n\in\mathbb{N}\}$.
The family of all finite subsets of $\mathbb{N}$ is denoted by $\mathcal{I}_{fin}$. Then $\mathcal{I}_{fin}$ is the smallest non-trivial ideal contained in each admissible ideal. The following, if no otherwise specified, we consider $\mathcal{I}$ is always an admissible ideal on the set $\mathbb{N}$.

\medskip
The concept of $\mathcal{I}$-convergence is a generalization of
the usual convergence in topological spaces. Let $X$ be a topological space and $\tau_{X}$ denote the topology for the space $X$. A sequence $\{x_{n}\}_{n\in\mathbb{N}}$ in $X$ is said to be {\it$\mathcal{I}$-eventually in} a subset $P\subseteq X$ if the set $\{n\in\mathbb{N}: x_{n}\notin P\}\in\mathcal{I}$ \cite[Definition 3.15]{cbb12}. A sequence $\{x_n\}_{n\in\mathbb{N}}$ in $X$ is said to be {\it$\mathcal{I}$-convergent} to a point $x\in X$ provided $\{x_n\}_{n\in\mathbb N}$ is $\mathcal{I}$-eventually in every neighborhood of $x$ in $X$, which is denoted by $x_{n} \xrightarrow {_{\mathcal{I}}} x$, and the point $x$ is called the $\mathcal{I}$-{\it limit point} of the sequence $\{x_n\}_{n\in\mathbb{N}}$. A subset $P$ of $X$ is said to be {\it$\mathcal{I}$-closed} if for each sequence $\{x_n\}_{n\in\mathbb{N}}\subseteq P$ with $x_{n} \xrightarrow {_{\mathcal{I}}} x\in X$, the {\it $\mathcal{I}$-limit point} $x\in P$.
A subset $P$ of $X$ is said to be {\it $\mathcal{I}$-open} if the complement set $X\setminus P$ is $\mathcal{I}$-closed.

\medskip
Let $\mathcal{I}$ be an ideal on $\mathbb{N}$. Let $f:X\to Y$ be a mapping. $f$ is called {\it preserving $\mathcal{I}$-convergence} provided $x_{n} \xrightarrow {_{\mathcal{I}}} x\in X$, then $f(x_{n}) \xrightarrow {_{\mathcal{I}}} f(x)\in Y$ \cite[Theorem 3]{ba2}.
It is easy to check that every continuous mapping preserves $\mathcal{I}$-convergence~\cite[Theorem 4.2]{cbb12}.
One of the mappings corresponding to preserving $\mathcal{I}$-convergence is an $\mathcal{I}$-covering mapping.

\begin{definition}
Let $\mathcal{I}$ be an ideal on $\mathbb{N}$. A mapping $f:X\to Y$ is called {\it $\mathcal{I}$-covering} provided a sequence $y_{n} \xrightarrow {_{\mathcal{I}}} y$ in $Y$, there is a sequence $x_{n} \xrightarrow {_{\mathcal{I}}} x$ in $X$ satisfying $x\in f^{-1}(y)$ and each $x_n\in f^{-1}(y_n)$ \cite[Definition 5.1]{cbb12}.
\end{definition}

The concept of networks has played a key role in the study of topological spaces. A family $\mathcal{P}$ of subsets of a topological space $X$ is called a {\it network} at a point $x\in X$ if $x\in\bigcap\mathcal{P}$ and whenever $x\in U$ with $U$ open in $X$, then $P\subseteq U$ for some $P\in\mathcal {P}$ \cite[Definition 1.5.9]{LsYzq16}.

\begin{definition}\label{d1}
Let $\mathcal{I}$ be an ideal on $\mathbb{N}$, and $\mathcal{P}$ a family of subsets of a topological space $X$.
$\mathcal{P}$ is called an {\it $\mathcal{I}$-cs-network} at a point $x\in X$ if whenever $\{x_n\}_{n\in\mathbb{N}}$ is a sequence $\mathcal{I}$-converging to $x\in U$ with $U$ open in $X$ then $\{x\}\cup\{x_n: n\in\mathbb{N}\setminus I\}\subseteq P \subseteq U$ for some $I\in\mathcal{I}$ and $P\in \mathcal{P}$.
$\mathcal{P}$ is called an {\it $\mathcal{I}$-cs-network} for $X$ if $\mathcal{P}$ is an {\it $\mathcal{I}$-cs-network} at each point $x\in X$ \cite[Definition 4.1]{zl202011}.
Each $\mathcal{I}_{fin}$-\emph{cs}-network is called a {\it $cs$-network} \cite[P.106]{Gu71}.
\end{definition}

It is easy to check that the statement ``$\{x_n: n\in\mathbb{N}\setminus I\}\subseteq P$ for some $I\in\mathcal{I}$'' is equivalent to the statement ``$\{n\in\mathbb N: x_n\not\in P\}\in\mathcal{I}$'' in Definition \ref{d1}, i.e., the sequence $\{x_n\}_{n\in\mathbb N}$ is $\mathcal{I}$-eventually in the set $P$.

\begin{definition}\label{d25}
Suppose that $\mathscr{P}$ is a family of subsets of a $T_1$-space $X$ {\color{black}such that, for each $x\in X$, there is a countable subfamily of $\mathscr{P}$ which is a network at $x$ in $X$}. Let $\mathscr{P}=\{P_{\alpha}:\alpha\in {\it \Lambda}\}$, {\color{black}which is no repetition by indexes in the enumeration.} ${\it \Lambda}$ is endowed with the discrete topology. Put
$$M=\{\alpha=(\alpha_{i})\in {\it \Lambda}^{\omega}:\{P_{\alpha_{i}}\}_{i\in \mathbb{N}}\mbox{~forms~a~network~at~some~point~}x_{\alpha}\mbox{~in~}X\}.$$
$M$ is endowed with the subspace topology of the product space ${\it\Lambda}^{\omega}$, and a function $f:M\rightarrow X$ is defined by $f(\alpha)=x_{\alpha}$ for every $\alpha\in M$. Then $(f,M,X,\mathscr{P})$ is called {\it Ponomarev's system}~\cite[p. 296]{LsYpf03}.
\end{definition}

It is known $M$ is a metric space and the function $f:M\to X$ is continuous and surjective ~\cite[p. 296]{LsYpf03}.
Ponomarev's system is one of the important methods to construct metric spaces and certain mappings on the metric spaces, and it is also a basic {\color{black} tool} to discuss the images of metric spaces under certain mappings \cite{Ls15, LsYzq16}.

\section{On spaces with $\mathcal{I}$-$csf$-networks}

Finding the intrinsic properties of the images of metric spaces under certain mappings are one of the basic topics in the study of topological spaces. In this section we characterized the images of metric spaces under continuous and $\mathcal{I}$-covering mappings. Next, we introduce a special countable network based on the $\mathcal{I}$-$cs$-networks.

\begin{definition}\label{d2}
Let $\mathcal{I}$ be an ideal on $\mathbb{N}$, and $\mathcal{P}$ a family of subsets of a topological space $X$.
$\mathcal{P}$ is called an {\it $\mathcal{I}$-csf-network} at a point $x\in X$ if whenever $\{x_n\}_{n\in\mathbb{N}}$ is a sequence $\mathcal{I}$-converging to $x\in X$, then there is a countable subfamily $\mathcal{P}'$ of $\mathcal{P}$ satisfying provided $x\in U$ with $U$ open in $X$ then $\{x\}\cup\{x_n: n\in\mathbb{N}\setminus I\}\subseteq P \subseteq U$ for some $I\in\mathcal{I}$ and $P\in \mathcal{P}'$.
$\mathcal{P}$ is called an {\it $\mathcal{I}$-csf-network} for $X$ if $\mathcal{P}=\bigcup_{x\in X}\mathcal{P}_x$ and each $\mathcal{P}_x$ is an $\mathcal{I}$-$csf$-network at $x$ in $X$.
\end{definition}

Each $\mathcal{I}_{fin}$-$csf$-network is called a {\it $csf$-network} \cite[Definition 2.5]{Ge05}.
It is known that every $\mathcal{I}$-$cs$-network is preserved by a continuous $\mathcal{I}$-covering mapping \cite[Theorem 4.3]{zl202011}.

\begin{lemma}\label{I-covering-pre}
Every $\mathcal{I}$-$csf$-network is preserved by a continuous $\mathcal{I}$-covering mapping.
\end{lemma}
\begin{proof}
Let $f:X\rightarrow Y$ be a continuous $\mathcal{I}$-covering mapping and $\mathcal{P}$ be an $\mathcal{I}$-$csf$-network for the topological space $X$. Put $\mathcal{P}=\bigcup_{x\in X}\mathcal{P}_x$ in which each $\mathcal{P}_x$ is an $\mathcal{I}$-$csf$-network at $x$ in $X$. Now, put $\mathcal{Q}=\bigcup_{y\in Y}\mathcal{Q}_y$, where each $\mathcal{Q}_y=\{f(P): P\in\mathcal{P}_x, x\in f^{-1}(y)\}$. Then $\mathcal{Q}=f(\mathcal{P})$. Suppose that $\{y_n\}_{n\in\mathbb{N}}$ is a sequence in $Y$, which is $\mathcal{I}$-convergent to a point $y\in Y$. Since $f$ is an $\mathcal{I}$-covering mapping, there exists a sequence $\{x_{n}\}_{n\in\mathbb{N}}$ in $X$ with $x_{n} \xrightarrow {_{\mathcal{I}}} x\in f^{-1}(y)$ such that each $x_{n}\in f^{-1}(y_{n})$. Since $\mathcal{P}_x$ is an $\mathcal{I}$-$csf$-network at $x$ in $X$, there is a countable subfamily $\mathcal{P}'$ of $\mathcal{P}_x$ satisfying the condition in Definition \ref{d2}. Put $\mathcal{Q}'=f(\mathcal{P}')$. Then $\mathcal{Q}'$ is a countable subfamily of $\mathcal{Q}_y$. If $y\in U$ with $U$ open in $Y$,  then $x\in f^{-1}(U)$ and $f^{-1}(U)$ is open in $X$. There are $I\in\mathcal{I}$ and $P\in\mathcal{P}'$ such that $\{x\}\cup\{x_n: n\in\mathbb{N}\setminus I\}\subseteq P \subseteq f^{-1}(U)$, and thus $f(P)\in\mathcal{Q}'$ and $\{y\}\cup\{y_n: n\in\mathbb{N}\setminus I\}\subseteq f(P) \subseteq U$. This means that the family $\mathcal{Q}$ is an $\mathcal{I}$-$csf$-network for $Y$.
\end{proof}

Perhaps, $\mathcal{I}$-covering mappings are one of the most appropriate mappings to adapt to $\mathcal{I}$-convergence in the relationship between spaces and mappings. However, a finite-to-one, continuous and closed mapping on a metric space is not necessarily an $\mathcal{I}$-covering mapping \cite[Example 3.5.17(1)]{Ls15}. The following result provides a technical lemma for deciding $\mathcal{I}$-covering mappings, in which its description is similar to the form of the usual convergence of sequences, but the proof is much more complex.

\begin{lemma}\label{decreasing network}
Let $\mathcal{I}$ be an ideal on $\mathbb{N}$. And let $f:X\rightarrow Y$ be a surjective mapping and $\{y_{i}\}_{i\in\mathbb{N}}$ a sequence $\mathcal{I}$-converging to some point $f(x)$ in $Y$. If $\{B_{n}\}_{n\in\mathbb{N}}$ is a decreasing network at $x$ in $X$ and the sequence $\{y_{i}\}_{i\in\mathbb{N}}$ is $\mathcal{I}$-eventually in $f(B_{n})$ for each $n\in\mathbb{N}$, then there is a sequence $\{x_i\}_{i\in\mathbb{N}}$ $\mathcal{I}$-converging to the point $x$ in $X$ with each $x_{i}\in f^{-1}(y_{i})$.
\end{lemma}
\begin{proof}
For each $n\in\mathbb{N}$, let $I_n=\{i\in\mathbb{N}: y_{i}\notin f(B_{n})\}$ and $F_n=\mathbb N\setminus I_n$. Then $I_{n}\in\mathcal{I}$, because the sequence $\{y_{i}\}_{i\in\mathbb{N}}$ is $\mathcal{I}$-eventually in $f(B_{n})$; and $i\in I_n$ (i.e., $i\not\in F_n$) if and only if $y_i\notin f(B_n)$. Since $\{B_{n}\}_{n\in\mathbb{N}}$ is decreasing in $X$, it follows that $I_{n}\subseteq I_{n+1}$ and $F_{n+1}\subseteq F_n$ for each $n\in\mathbb{N}$.

\medskip
\textbf{Claim.} For each $k\in\mathbb N$, there exists $n_k\in\mathbb N$ such that if $x_{i}\notin B_{k}$ then $y_{i}\notin f(B_{n_k})$.

(a)\, Suppose that there is some $n'>n$ such that $F_{n'}\subset F_n$ for each $n\in\mathbb N$. Take a sequence $\{n_k\}_{k\in\mathbb N}$ in $\mathbb N$ such that each $n_{k}<n_{k+1}$ and $F_{n_{k+1}}\subset F_{n_k}$. Then $F_{n_0}=\bigcup_{k\in\mathbb N}(F_{n_k}\setminus F_{n_{k+1}})$.
For each $i\in\mathbb{N}$, if $i\in F_{n_k}$, then $y_i\in f(B_{n_k})$, thus $f^{-1}(y_i)\cap B_{n_k}\neq\varnothing$. We can pick
$$x_i\in\left\{
\begin{array}{ll}
f^{-1}(y_{i}), & i\in I_{n_{0}},\\
f^{-1}(y_{i})\cap B_{n_k},& i\in F_{n_k}\setminus F_{n_{k+1}},  k\in\mathbb{N}.
\end{array}
\right.$$

Let $x_i\notin B_k$.  If $i\in I_{n_0}$, then $y_i\notin f(B_{n_0})\supseteq f(B_{n_k})$, and $y_i\notin f(B_{n_k})$. If $i\in F_{n_0}$, there is $k'\in\mathbb N$ such that $i\in F_{n_{k'}}\setminus F_{n_{k'+1}}$, thus $x_i\in B_{n_{k'}}\setminus B_k\subseteq B_{k'}\setminus B_k$, and $k'<k$, because $\{B_{n}\}_{n\in\mathbb{N}}$ is decreasing. It follows from $i\notin F_{n_{k'+1}}$ that $y_{i}\notin f(B_{n_{k'+1}})\supseteq f(B_{n_k})$, thus $y_{i}\notin f(B_{n_{k}})$.

\medskip
(b)\, Suppose that there is some $n_{0}\in\mathbb{N}$ such that $F_n=F_{n_{0}}$ for each $n>n_0$. Let $n_k=n_0+k$ for each $k\in\mathbb N$. Since $\mathcal{I}$ is admissible, the set $F_{n_0}$ is infinite. Put $F_{n_0}=\{m_k\in\mathbb N: k\in\mathbb N\}$. If $i=m_k$, then $i\in F_{n_0}=F_{n_k}$, thus $y_i\in f(B_{n_k})$, hence $f^{-1}(y_i)\cap B_{n_k}\neq\varnothing$. We can pick
$$x_i\in\left\{
\begin{array}{ll}
f^{-1}(y_{i}), & i\in I_{n_{0}},\\
f^{-1}(y_{i})\cap B_{n_k},& i=m_k,  k\in\mathbb{N}.
\end{array}
\right.$$

If $x_{i}\notin B_{k}$, then $x_i\notin B_{n_k}$, thus $i\notin F_{n_0}=F_{n_k}$, hence $y_i\notin f(B_{n_k})$.

\medskip
Next, we will show that $x_{i} \xrightarrow {_\mathcal{I}} x$ in $X$.
Let $U$ be a neighborhood of $x$ in $X$. There exists $k\in\mathbb{N}$ such that $x\in B_{k}\subseteq U$. Therefore
$\{i\in\mathbb{N}: x_{i}\notin U\}\subseteq \{i\in\mathbb{N}: x_{i}\notin B_{k}\}\subseteq\{i\in\mathbb{N}: y_{i}\notin f(B_{n_k})\}=I_{n_k}\in\mathcal{I}$, hence $\{i\in\mathbb{N}: x_{i}\notin U\}\in\mathcal{I}$, and further $x_{i} \xrightarrow {_\mathcal{I}} x$ in $X$.
\end{proof}

By Lemma \ref{decreasing network}, the following corollary is obvious.

\begin{corollary}
Every open and surjective mapping on a first-countable space is an $\mathcal{I}$-covering mapping for each ideal $\mathcal{I}$ on $\mathbb N$.
\end{corollary}

The following is the main result in this paper, in which the space $X$ is not assumed to satisfy any separation axiom.

\begin{theorem}\label{main}
Let $\mathcal{I}$ be an ideal on $\mathbb{N}$. A space $X$ has an $\mathcal{I}$-$csf$-network if and only if $X$ is a continuous and $\mathcal{I}$-covering image of a metric space.
\end{theorem}
\begin{proof}
Sufficiency. Suppose that there exist a metric space $M$ and a continuous and $\mathcal{I}$-covering mapping $f:M\to X$. Since each point of $M$ has a countable local base, by Lemma \ref{I-covering-pre}, the space $X$ has an $\mathcal{I}$-$csf$-network.

\medskip
Necessity. Suppose that a space $X$ has an $\mathcal{I}$-$csf$-network. We will construct a metric space $M$ and the required mapping $f$ on $M$ in the following steps.

\medskip
\textbf{Claim~1.} There are a metric space $Y$ and a $T_2$-subspace $Z\subseteq X\times Y$ such that the restriction mapping $\pi_{1\mid Z}: Z\to X$ is continuous and surjective, where $\pi_1: X\times Y\to X$ is the projective mapping.

Let $Y$ be the product space $\prod_{n\in\mathbb N}Y_n$, where each space $Y_n$ is the set $X$ endowed with the discrete topology. Then $Y$ is a metric space. For each $x\in X$, put
$$Y_x=\{(y_n)\in Y: y_n=x\mbox{~except for finite~}n\in\mathbb N\}.$$
Then the family $\{Y_x: x\in X\}$ is disjoint. Put
$$Z=\bigcup\{\{x\}\times Y_x: x\in X\}.$$
It is obvious that the restriction mapping $\pi_{1\mid Z}: Z\to X$ is continuous and surjective. Since the family $\{Y_x: x\in X\}$ is disjoint, the restriction mapping $\pi_{2\mid Z}: Z\to Y$ is continuous and injective, and $Z$ is a $T_2$-space.

Let $\mathcal{P}$ be an $\mathcal{I}$-$csf$-network for the space $X$, and $\mathcal{Q}$ be a point-countable base for the metric space $Y$. Put $\mathcal{R}=(\mathcal{P}\times\mathcal{Q})_{\mid Z}=\{(P\times Q)\cap Z: P\in\mathcal{P}, Q\in\mathcal{Q}\}$, and denote $\mathcal{R}=\{R_\alpha: \alpha\in\mathit{\Lambda}\}$.
Let $(g, M, Z, \mathcal{R})$ be Ponomarev's system in Definition \ref{d25}. Then the mapping $g: M\to Z$ is continuous and surjective.

\medskip
For each $\alpha=(\alpha_{n})\in M$ and $k\in\mathbb{N}$, put
$$B_{k}=\{(\beta_{n})\in M: \beta_{n}=\alpha_{n}\mbox{~for each~}n\leq k\}.$$

\medskip
\textbf{Claim~2.} $g(B_{k})=\bigcap_{n\leq k}R_{\alpha_{n}}$.

Suppose that a point $\beta=(\beta_{n})\in B_{k}$. Then $$g(\beta)\in\bigcap_{n\in \mathbb{N}} R_{\beta_{n}}\subseteq \bigcap_{n\leq k}R_{\beta_{n}}=\bigcap_{n\leq k}R_{\alpha_n},$$
hence $g(B_{k})\subseteq\bigcap_{n\leq k}R_{\alpha_{n}}$. On the other hand, assume that a point $z\in\bigcap_{n\leq k}R_{\alpha_{n}}$. Since $g$ is surjective, there exists $\gamma=(\gamma_n)\in M$ with $g(\gamma)=z$, i.e., the subfamily $\{R_{\gamma_{n}}\}_{n\in\mathbb N}$ of $\mathcal{R}$ is a network at $z$ in $Z$. For each $n\in\mathbb N$, define $\beta_n\in\mathit{\Lambda}$ such that $\beta_n=\alpha_n$ if $n\leq k$ and $\beta_{n}=\gamma_{n-(k+1)}$ if $n>k$. Then the family $\{R_{\beta_n}\}_{n\in\mathbb N}$ is also a network at $z$ in $Z$.
Put $\beta=(\beta_{n})\in \mathit{\Lambda}^{\omega}$. Then $\beta\in B_k$ and $z=g(\beta)\in g(B_{k})$, and further $\bigcap_{n\leq k}R_{\alpha_{n}}\subseteq g(B_{k})$. Therefore, $g(B_{k})=\bigcap_{n\leq k}R_{\alpha_{n}}$.

\medskip
\textbf{Claim~3.} $g$ is an $\mathcal{I}$-covering mapping.

Let $\{z_{m}\}_{m\in\mathbb{N}}$ be a sequence $\mathcal{I}$-converging to $z$ in $Z$. Put $z=(x, y)$ and each $z_m=(x_m, y_m)$. Since the projective mappings $\pi_1$ and $\pi_2$ are continuous, they preserve $\mathcal{I}$-convergence. Then $x_m \xrightarrow {_{\mathcal{I}}}x$ in $X$ and $y_m\xrightarrow{_{\mathcal{I}}}y$ in $Y$. Let $\mathcal{P}_x=\{P_{x, i}\}_{i\in\mathbb N}\subseteq\mathcal{P}$ be a countable network at $x$ in $X$ such that the sequence $\{x_m\}_{m\in\mathbb N}$ is $\mathcal{I}$-eventually in each $P_{x, i}$, and let $\mathcal{Q}_y=\{Q_{y, j}\}_{j\in\mathbb N}\subseteq\mathcal{Q}$ be a countable local base at $y$ in $Y$, in which the sequence $\{y_m\}_{m\in\mathbb N}$ is $\mathcal{I}$-eventually in each $Q_{y, j}$, because the set $Q_{y, j}$ is open in $Y$. Since $\{(P_{x, i}\times Q_{y, j})\cap Z: i, j\in\mathbb N\}$ is a network at $z$ in $Z$, there is $\alpha=(\alpha_n)\in M$ such that $g(\alpha)=z$ and $\{R_{\alpha_n}: n\in\mathbb N\}=\{(P_{x, i}\times Q_{y, j})\cap Z: i, j\in\mathbb N\}$. Denote $R_{\alpha_n}=(P_{x, i_{\alpha_n}}\times Q_{y, j_{\alpha_n}})\cap Z$ for each $n\in\mathbb N$.

For the above $\alpha=(\alpha_n)\in M$ and each $k\in\mathbb N$, by Claim 2,
\begin{align*}
\{m\in\mathbb N: z_m\not\in g(B_k)\}&=\bigcup_{n\leq k}\{m\in\mathbb N: z_m\not\in R_{\alpha_n}\}\\
&=\bigcup_{n\leq k}(\{m\in\mathbb N: x_m\not\in P_{x, i_{\alpha_n}}\}\cup\{m\in\mathbb N: y_m\not\in Q_{y, j_{\alpha_n}}\})\in\mathcal{I}
\end{align*}
This implies that the sequence $\{z_{m}\}_{m\in\mathbb{N}}$ is $\mathcal{I}$-eventually in $g(B_k)$. It is obvious that the family $\{B_{k}\}_{k\in\mathbb{N}}$ is a decreasing local base at $\alpha$ in $M$. In the view of  Lemma \ref {decreasing network}, $g$ is an $\mathcal{I}$-covering mapping.

\medskip
Finally, put $f=\pi_{1\mid Z}\circ g:M\to X$. Then $f$ is continuous and surjective. Let $\{x_{m}\}_{m\in\mathbb{N}}$ be a sequence in $X$ with $x_{m} \xrightarrow {_{\mathcal{I}}} x$. Put $y=(y_n)\in Y$ with $y_n=x$ for each $n\in\mathbb N$. Then $y\in Y_x$, thus $(x, y)\in Z$. For each $m\in\mathbb N$, define $v_m=(v_{m,n})\in Y$ as $v_{m,n}=x$ if $n\leq m$ and $v_{m,n}=x_m$ if $n>m$; then $v_m\in Y_{x_m}$, thus $(x_m, v_m)\in Z$. It is easy to see that the sequence $\{v_{m}\}_{m\in\mathbb N}$ converges to the point $y$ in $Y$. Let $O$ be a neighborhood of the point $(x, y)$ in $Z$. Take an open subset $U$ in $X$ and an open subset $V$ in $Y$ such that $(x, y)\in (U\times V)\cap Z\subseteq O$. Then $\{m\in\mathbb N: (x_m, v_m)\not\in O\}\subseteq\{m\in\mathbb N: (x_m, v_m)\not\in (U\times V)\cap Z\}=\{m\in\mathbb N: x_m\not\in U\}\cup\{m\in\mathbb N: v_m\not\in V\}\in\mathcal{I}$. Thus the sequence $\{(x_{m}, v_m)\}_{m\in\mathbb N}$ is $\mathcal{I}$-convergent to $(x, y)$ in $Z$. By Claim 3, there exists a sequence $\{z_{m}\}_{m\in\mathbb{N}}$ in $M$, which satisfies $z_{m} \xrightarrow {_{\mathcal{I}}} z\in g^{-1}((x, y))\subseteq f^{-1}(x)$ and each $z_{m}\in g^{-1}((x_{m}, v_m))\subseteq f^{-1}(x_m)$. Thus, $f$ is an $\mathcal{I}$-covering mapping.
\end{proof}

At the end of this section, we give an application of the proving method in Theorem \ref{main}. A family $\mathcal{P}$ of subsets of a set $X$ is called {\it point-countable} if each point of $X$ belongs to at least countable elements of the family $\mathcal{P}$. A mapping $f:X\to Y$ is an {\it $s$-mapping} if $f^{-1}(y)$ is a separable subset of $X$ for each $y\in Y$.

\begin{corollary}\label{c31}
Let $\mathcal{I}$ be an ideal on the set $\mathbb N$. Then a $T_1$-space $X$ has a point-countable $\mathcal{I}$-$cs$-network if and only if $X$ is the image of a metric space under a continuous $\mathcal{I}$-covering and $s$-mapping.
\end{corollary}

\begin{proof}
Let $X$ be a $T_1$-space with a point-countable $\mathcal{I}$-$cs$-network $\mathcal{R}$. Let $(g, M, X, \mathcal{R})$ be Ponomarev's system. It follows from Claims 2 and 3 in the proof of Theorem \ref{main} that the mapping $g: M\to X$ is a continuous and $\mathcal{I}$-covering mapping. Put $\mathcal{R}=\{R_\alpha: \alpha\in\mathit{\Lambda}\}$. If $x\in X$, then
\begin{align*}
g^{-1}(x)&=\{(\alpha_n)\in M: \{R_{\alpha_n}\}_{n\in\mathbb N}\mbox{~forms a network at the point~}x\in X\}\\
&\subseteq\{\alpha\in\mathit{\Lambda}: x\in R_\alpha\}^\omega
\end{align*}
thus $g^{-1}(x)$ is a separable subset of $M$. Hence, $g$ is an $s$-mapping.

On the other hand, let $f:M\to X$ be a continuous $\mathcal{I}$-covering and $s$-mapping, where $M$ is a metric space. Let $\mathcal{B}$ be a point-countable base for $M$. By Lemma \ref{I-covering-pre}, it is easy to check that the family $\{f(B): B\in\mathcal{B}\}$ is an $\mathcal{I}$-$cs$-network for $X$. Since every point-countable family of open subsets of a separable space is countable, the family $\{f(B): B\in\mathcal{B}\}$ is point-countable. Thus, $X$ has a point-countable $\mathcal{I}$-$cs$-network.
\end{proof}

\section{$\mathcal{I}$-$csf$-countable spaces}

A general space than a space with a point-countable $\mathcal{I}$-$cs$-network is the following $\mathcal{I}$-$csf$-countable space. A space $X$ is called an {\it $\mathcal{I}$-$csf$-countable space} if, $X$ has a countable $\mathcal{I}$-$cs$-network at each point in $X$. Each $\mathcal{I}_{fin}$-\emph{csf}-countable space is called a {\it $csf$-countable space} \cite[p. 181]{LY01}. It is obvious that every first-countable space is $\mathcal{I}$-$csf$-countable, and every $\mathcal{I}$-$csf$-countable space has an $\mathcal{I}$-$csf$-network.

A mapping $f:X\to Y$ is a {\it boundary $s$-mapping} if $\partial f^{-1}(y)$ is a separable subset of $X$ for each $y\in Y$.

\begin{theorem}\label{csf}
Let $\mathcal{I}$ be an ideal on the set $\mathbb N$.  Then a $T_1$-space $X$ is an $\mathcal{I}$-$csf$-countable space if and only if $X$ is the image of a metric space under a continuous $\mathcal{I}$-covering and boundary $s$-mapping.
\end{theorem}

\begin{proof}
Necessity. Suppose that $X$ is an $\mathcal{I}$-$csf$-countable $T_1$-space. For each $x\in X$, let $X_{x}$ be the set $X$ endowed with the following topology: a neighborhood base of $x$ in $X_x$ is taken as the neighborhood base of $x$ in the original topology of $X$; every point of $X_x\setminus\{x\}$ is an isolated point. Put $Y=\bigoplus_{x\in X}X_x$. Define a function $h:Y\rightarrow X$ by the natural function, i.e., $h|_{X_x}=\mbox{id}_X$ for each $x\in X$.

\medskip
{\bf Claim 1.}\quad $Y$ has a point-countable $\mathcal{I}$-$cs$-network.

For each $x\in X$, let $\mathscr{P}_x$ be a countable $\mathcal{I}$-$cs$-network at $x$ in $X$. For each $y\in Y$, there exists a unique $x\in X$ such that $y\in X_x$. If $y=x$, let $\mathscr{F}_{y}=\mathscr{P}_{x}$; if $y\neq x$, let {\color{black} $\mathscr{F}_{y}=\{\{y\}\}$}. Put $\mathscr{F}=\bigcup_{y\in Y}\mathscr{F}_{y}$. It is easy to see that $\mathscr{F}$ is a point-countable $\mathcal{I}$-$cs$-network for $Y$.

\medskip
{\bf Claim 2.}\quad $h$ is a continuous and $\mathcal{I}$-covering mapping satisfying $\partial h^{-1}(x)\subseteq\{x\}$ for each $x\in X$.

Obviously, $h$ is continuous. Let $\{x_n\}_{n\in\mathbb N}$ be a sequence $\mathcal{I}$-converging to a point $x$ in $X$. Then the sequence $\{x_n\}_{n\in\mathbb N}$ is also $\mathcal{I}$-converging to $x$ in $X_x$. It is obvious that $x\in h^{-1}(x)\cap X_{x}\subseteq Y$ and each $h(x_n)=x_n$. Thus $h$ is an $\mathcal{I}$-covering mapping. For each $x\in X$ and $y\in X\setminus\{x\}$, since $X$ is a $T_1$-space, the set $h^{-1}(x)\cap X_{y}=\{y\}$ is closed and open in $Y$, and so $\partial h^{-1}(x)\subseteq\{x\}$.

\medskip
Since $Y$ is a $T_1$-space with a point-countable $\mathcal{I}$-$cs$-network, by Corollary \ref{c31}, there are a metric space $M$ and a continuous $\mathcal{I}$-covering $s$-mapping $g:M\to Y$.

\medskip
{\bf Claim 3.}\quad $f=h\circ g:M\rightarrow X$ is a continuous $\mathcal{I}$-covering and boundary $s$-mapping.

It is clear that $f$ is a continuous and $\mathcal{I}$-covering mapping. For each $x\in X$, since $g^{-1}([h^{-1}(x)]^{\circ})$ is open in $M$,
\begin{align*}
\partial f^{-1}(x)&=\partial (g^{-1}(h^{-1}(x)))\\
&=g^{-1}(h^{-1}(x))\setminus [g^{-1}(h^{-1}(x))]^{\circ}\\
&\subseteq g^{-1}(h^{-1}(x))\setminus g^{-1}([h^{-1}(x)]^{\circ})=g^{-1}(\partial h^{-1}(x)).
\end{align*}
By Claim 2, the set $\partial f^{-1}(x)$ is a separable set in $M$. So $f$ is a boundary $s$-mapping.

\medskip
Sufficiency. Suppose that there are a metric space $M$ and a continuous $\mathcal{I}$-covering boundary $s$-mapping $f:M\to X$. Let $\mathscr{B}$ be a point-countable base for $M$. If $x\in X$ and $\{x\}$ is not open in $X$, then $\partial f^{-1}(x)\neq\emptyset$, and pick $m_x\in\partial f^{-1}(x)$. Put
$$\mathscr{P}_x=\{f(B): B\in\mathscr{B}\mbox{~and~}B\cap\partial f^{-1}(x)\neq \varnothing\}.$$
Since the set $\partial f^{-1}(x)$ is separable, the family $\mathscr{P}_x$ is countable. Let $\{x_{i}\}_{i\in\mathbb{N}}$ be a sequence in $X$, $\mathcal{I}$-converging to the point $x$ and $x\in U\in \tau_{X}$. If there is $I\in\mathcal{I}$ such that $x_i=x$ for each $i\in\mathbb N\setminus I$, we take $B\in\mathcal{B}$ with $m_x\in B\subseteq f^{-1}(U)$, then $f(B)\in\mathscr{P}_x$ and $\{x\}\cup\{x_{i}:i\in\mathbb N\setminus I\}=\{x\}\subseteq f(B)\subseteq U$. If there is no $I\in\mathcal{I}$ such that $x_i=x$ for each $i\in\mathbb N\setminus I$, since $f$ is $\mathcal{I}$-covering, there is a sequence $\{y_{i}\}_{i\in\mathbb{N}}$ in $M$, $\mathcal{I}$-converging to a point $y\in f^{-1}(x)$ with each $y_{i}\in f^{-1}(x_{i})$. Then $y\in\partial f^{-1}(x)$. Otherwise, $y\in [f^{-1}(x)]^\circ$, thus there is $J\in\mathcal{I}$ such that $\{y_i: i\in\mathbb N\setminus J\}\subseteq [f^{-1}(x)]^\circ$, so $x_i=x$ for each $i\in\mathbb N\setminus J$, which is a contradiction. This means that $y\in\partial f^{-1}(x)$. Then $y\in f^{-1}(x)\subseteq f^{-1}(U)\in\tau_M$, and there exists $B\in\mathscr{B}$ such that $y\in B\subseteq f^{-1}(U)$. As a consequence, $y\in B\cap\partial f^{-1}(x)$ and there is $I\in\mathcal{I}$ such that $\{y_{i}:i\in\mathbb N\setminus I\}\subseteq B$. It follows that $f(B)\in\mathscr{P}_x$ and $\{x\}\cup\{x_{i}:i\in\mathbb N\setminus I\}\subseteq f(B)\subseteq U$. Therefore, $\mathscr{P}_x$ is a countable $\mathcal{I}$-$cs$-network at $x$ in $X$.
\end{proof}

\begin{theorem}
Every $\mathcal{I}$-$csf$-countable space is a $csf$-countable space.
\end{theorem}

\begin{proof}
Let $X$ be an $\mathcal{I}$-$csf$-countable space. For each $x\in X$, let $\mathscr{P}_{x}$ be a countable $\mathcal{I}$-$cs$-network at $x$ in $X$. Put $\mathscr{F}_{x}=\{\bigcup\mathscr{P}^{\prime}_{x}:\mathscr{P}^{\prime}_{x}\subseteq\mathscr{P}_{x}\mbox{~and~} |\mathcal{P}'_x|<\omega\}$. Then $\mathscr{F}_{x}$ is countable. We will show that $\mathscr{F}_x$ is a $cs$-network at $x$ in $X$. Suppose that a sequence $\{x_n\}_{n\in\mathbb N}$ converges to the point $x\in V$ with $V$ open in $X$. Put $\{F\in\mathscr{F}_x: F\subseteq V\}=\{F_i\}_{i\in\mathbb N}$. Then there exists  $k\in\mathbb N$ such that the sequence $\{x_n\}_{n\in\mathbb N}$ is eventually in $\bigcup_{i\leqslant k}F_i$. Otherwise, there exists a subsequence $\{x_{n_k}\}_{k\in\mathbb N}$ of the sequence $\{x_n\}_{n\in\mathbb N}$ such that each $x_{n_k}\in X\setminus\bigcup_{i\leqslant k}F_i$. Since the subsequence $\{x_{n_k}\}_{k\in\mathbb N}$ converges to $x$ and $\mathscr{P}_x$ is an $\mathcal{I}$-$cs$-network at $x$, there are $I\in\mathcal{I}$ and $P\in\mathcal{P}_x$ such that $\{x_{n_k}: k\in\mathbb N\setminus I\}\subseteq P\subseteq V$. By $P\in\mathscr{F}_x$, we have $P=F_m$ for some $m\in\mathbb N$. Since $\mathbb N\setminus I$ is infinite, there is $k_0\in\mathbb N\setminus I$ with $k_0\geq m$, thus $x_{n_{k_0}}\not\in F_m=P$, which is a contradiction. Therefore, $X$ is a $csf$-countable space.
\end{proof}

\section{Several applications}

In this section, we discuss the preliminary applications of the main theorems and put forward several related questions.

Let $\mathcal{I}$ be an ideal on the set $\mathbb N$. A subset $P$ of a topological space $X$ is said to be an {\it$\mathcal{I}_{sn}$-open set} of $X$ provided each sequence in $X$ $\mathcal{I}$-converging to a point $x\in P$ is $\mathcal{I}$-eventually in $P$ \cite[P. 1982]{Lin2020}. We have that open subsets $\Longrightarrow\mathcal{I}_{sn}$-open subsets $\Longrightarrow\mathcal{I}$-open subsets $\Longrightarrow$ sequentially open subsets in a topological space \cite[Lemma 2.1]{Lin2020}. Here, $\mathcal{I}_{fin}$-open subsets are called {\it sequentially open}.

\begin{definition}
Let $\mathcal{I}$ be an ideal on $\mathbb N$. A topological space $X$ is called an {\it $\mathcal{I}$-FU-space} provided $A\subseteq X$ and $x\in\overline{A}$ there is a sequence $\{x_n\}_{n\in\mathbb N}$ in $A$ with $x_{n}\xrightarrow{_{\mathcal{I}}}x$ in $X$ \cite[P. 90]{RePr16}; $X$ is called an {\it $\mathcal{I}$-sequential space} if each $\mathcal{I}$-open subset of $X$ is open~\cite[Definition 2.3]{Pa14}; $X$ is called an {\it $\mathcal{I}$-neighborhood space} if each $\mathcal{I}$-open subset of $X$ is $\mathcal{I}_{sn}$-open~\cite[Definition 3.1]{Lin2020}.
\end{definition}

An $\mathcal{I}_{fin}$-FU-space is called a {\it Fr\'echet-Urysohn space} \cite[Definition 1.2.7]{LsYzq16}; an $\mathcal{I}_{fin}$-sequential space is called a {\it sequential space} \cite[Definition 1.6.15]{LsYzq16}; every topological space is an $\mathcal{I}_{fin}$-neighborhood space \cite[Example 3.11]{Lin2020}. It is easy to check that first-countable spaces $\Longrightarrow$ Fr\'echet-Urysohn spaces $\Longrightarrow\mathcal{I}$-FU-spaces $\Longrightarrow\mathcal{I}$-sequential spaces $\Longrightarrow \mathcal{I}$-neighborhood spaces~\cite[Lemma 3.4]{Lin2020}; and Fr\'echet-Urysohn spaces $\Longrightarrow$ sequential spaces $\Longrightarrow\mathcal{I}$-sequential spaces~\cite[Lemma 2.5]{Lin2020}.

\begin{corollary}
Let $\mathcal{I}$ be an ideal on $\mathbb{N}$. Then each space of $\mathcal{I}$-$csf$-networks is an $\mathcal{I}$-neighborhood space.
\end{corollary}

\begin{proof}
Let $X$ be a space with an $\mathcal{I}$-$csf$-network. By Theorem \ref{main}, there are a metric space $M$ and a continuous and $\mathcal{I}$-covering mapping $f:M\to X$. Let $U$ be an $\mathcal{I}$-open set in $X$. Then $f^{-1}(U)$ is $\mathcal{I}$-open in $M$. In fact, let $\{z_{n}\}_{n\in\mathbb N}$ be a sequence in $M\setminus f^{-1}(U)$ with $z_{n} \xrightarrow {_{\mathcal{I}}} z\in M$. Since $f$ is continuous, $f$ preserves $\mathcal{I}$-convergence. Thus, we have $f(z_{n}) \xrightarrow {_{\mathcal{I}}} f(z)$. Since the set $X\setminus U$ is $\mathcal{I}$-closed in $X$ and each $f(z_{n})\in X\setminus U$, therefore $f(z)\in X\setminus U$, i.e., $z\in M\setminus f^{-1}(U)$. Hence $M\setminus f^{-1}(U)$ is $\mathcal{I}$-closed in $X$, i.e., the set $f^{-1}(U)$ is $\mathcal{I}$-open in $M$, thus $f^{-1}(U)$ is open in $M$, because $M$ is a metric space.

Next, we show that $U$ is an $\mathcal{I}_{sn}$-open subset of $X$. Let $\{x_{n}\}_{n\in\mathbb{N}}$ be a sequence in $X$ with $x_{n} \xrightarrow {_{\mathcal{I}}} x\in U$. Since $f$ is $\mathcal{I}$-covering, there exists a sequence $\{z_{n}\}_{n\in\mathbb{N}}$ in $M$ satisfying $z_{n}\xrightarrow{_{\mathcal{I}}} z\in f^{-1}(x)$ and each $z_{n}\in f^{-1}(x_{n})$. By $z_{n} \xrightarrow {_{\mathcal{I}}} z\in f^{-1}(U)$, $\{n\in\mathbb{N}: x_{n}\notin U\}=\{n\in\mathbb{N}: z_{n}\notin f^{-1}(U)\}\in \mathcal{I}$, therefore the sequence $\{x_n\}_{n\in\mathbb N}$ is $\mathcal{I}$-eventually in $U$. This implies the set $U$ is an $\mathcal{I}_{sn}$-open subset of $X$. Thus, $X$ is an $\mathcal{I}$-neighborhood space.
\end{proof}

Let $f:X\to Y$ be a mapping. $f$ is called {\it quotient} provided $f$ is surjective and a subset $U$ of $Y$ is open if and only if $f^{-1}(U)$ is open in $X$ \cite[Definition 2.1.1]{LsYzq16}; $f$ is called {\it pseudo-open} provided $y\in Y$ and $f^{-1}(y)\subseteq U$ with $U$ open in $X$, then $f(U)$ is a neighborhood of $y$ in $Y$ \cite[Definition 1]{Ar63}. It is known that every continuous and pseudo-open mapping is quotient.

\begin{corollary}\label{c52}
Let $\mathcal{I}$ be an ideal on $\mathbb{N}$. The following are equivalent for a space $X$.

$(1)$\, $X$ is a sequential space of $\mathcal{I}$-$csf$-networks.

$(2)$\, $X$ is an $\mathcal{I}$-sequential space of $\mathcal{I}$-$csf$-networks.

$(3)$\, $X$ is an $\mathcal{I}$-covering and quotient image of a metric space.
\end{corollary}

\begin{proof}
Since every sequential space is preserved by a quotient mapping \cite[Proposition 2.3.1]{LsYzq16}, by Theorem \ref{main}, we have that $(3)\Rightarrow (1)$. It is obvious that $(1)\Rightarrow (2)$. Next, we show that $(2)\Rightarrow (3)$. Let $X$ be an $\mathcal{I}$-sequential space of $\mathcal{I}$-$csf$-networks. By Theorem \ref{main}, there are a metric space $M$ and a continuous and  $\mathcal{I}$-covering mapping $f:M\to X$. Suppose that $U\subseteq X$ and $f^{-1}(U)$ is open in $M$. If a sequence $x_n\xrightarrow {_{\mathcal{I}}} x\in U$ in $X$, then there is a sequence $z_n\xrightarrow {_{\mathcal{I}}} z\in f^{-1}(x)$ in $M$ with each $z_n\in f^{-1}(x_n)$. Since $z\in f^{-1}(U)$, the set $\{n\in\mathbb N: x_n\not\in U\}=\{n\in\mathbb N: z_n\not\in f^{-1}(U)\}\in\mathcal{I}$, i.e., $U$ is $\mathcal{I}$-open in the $\mathcal{I}$-sequential space $X$, thus $U$ is open. Therefore, $f$ is a quotient mapping.
\end{proof}

Similarly, we have the following corollary. In its proof, the following results are used: (a)\, every Fr\'echet-Urysohn space is preserved by a continuous and pseudo-open mapping \cite[Proposition 2.3.1]{Ls15}; (b)\, every $\mathcal{I}$-covering mapping onto an $\mathcal{I}$-FU-space is pseudo-open \cite[Theorem 6.7]{cbb12}.

\begin{corollary}\label{c53}
Let $\mathcal{I}$ be an ideal on $\mathbb{N}$. The following are equivalent for a space $X$.

$(1)$\, $X$ is a Fr\'echet-Urysohn space of $\mathcal{I}$-$csf$-networks.

$(2)$\, $X$ is an $\mathcal{I}$-FU-space of $\mathcal{I}$-$csf$-networks.

$(3)$\, $X$ is a continuous, $\mathcal{I}$-covering and pseudo-open image of a metric space.
\end{corollary}

Statistical convergence is a special ideal convergence \cite{cbb12}. Corollaries \ref{c52} and \ref{c53} partially answer the following questions, which were posed by Z.B Tang and F.C. Lin in \cite[Questions 2.1 and 3.1]{TangL}:

\medskip
(1)\, How to characterize $s$-sequential spaces (i.e., statistical sequential spaces) as the images of metric spaces under some continuous mappings?

(2)\, How to characterize statistical FU-spaces as the images of metric spaces under some continuous mappings?

\begin{example}\label{ex}
Every space has a $csf$-network. But, there are an ideal $\mathcal{I}$ on $\mathbb N$ and an $\mathcal{I}$-FU-space $X$ which has no $\mathcal{I}$-$csf$-network.
\end{example}

\begin{proof}
First, we show that every space has a $csf$-network. Let $X$ be a topological space and $x\in X$. If $\{x_n\}_{n\in\mathbb N}$ is a sequence with $x_n\to x$ in $X$. Put $\mathcal{P}_x=\{\{x\}\cup\{x_n: n\geq k\}: k\in\mathbb N\}$. Then $\mathcal{P}_x$ is countable. If $x\in U$ with $U$ open in $X$, then there exist  $I\in\mathcal{I}_{fin}$ and $k\in\mathbb N$ such that $\{x\}\cup\{x_n: n\in\mathbb N\setminus I\}=\{x\}\cup\{x_n: n\geq k\}\subseteq U$. Thus, $X$ has a $csf$-network.

Let $\mathcal{I}$ be a maximal ideal on $\mathbb N$. $\Sigma(\mathcal{I})$ is the set $\mathbb N\cup\{\infty\}$, $\infty\not\in\mathbb N$, equipped with the following topology: (a)\, each point $n\in\mathbb N$ is isolated; (b)\, each open neighborhood $U$ of $\infty$ is of the form $(\mathbb N\setminus I)\cup\{\infty\}$, for each $I\in\mathcal{I}$.

By \cite[Example 3.17]{Lin2020}, the space $\Sigma(\mathcal{I})$ is an $\mathcal{I}$-FU-space having no non-trivial convergent sequence. Since the point $\infty$ is non-isolated, $\Sigma(\mathcal{I})$ is not a sequential space. By Corollary \ref{c52}, $\Sigma(\mathcal{I})$ has no $\mathcal{I}$-$csf$-network.
\end{proof}

Let $X$ be a non-sequential space. For example, take $X=[0, \omega_1]$ with the usual ordered topology. Then $X$ is a non-$\mathcal{I}_{fin}$-sequential space
having an $\mathcal{I}_{fin}$-$csf$-network by Example \ref{ex}.

\begin{question}
Is there a Fr\'echet-Urysohn space having no $\mathcal{I}$-$csf$-network?
\end{question}

It is known that metrizability is preserved by continuous, closed and sequence-covering mapping \cite[Corollary 3.5.12]{Ls15}. V. Renukadevi and B. Prakash defined the statistically sequence covering map as follows \cite{RePr17}: a mapping $f:X\to Y$ is a {\it statistically sequence covering map} if whenever a sequence $\{y_n\}_{n\in\mathbb N}$ converges to a point $y$ in $Y$, there is a sequence $\{x_n\}_{n\in\mathbb N}$ statistically converging to a point $x$ in $X$ with each $x_n\in f^{-1}(y_n)$ and $x\in f^{-1}(y)$. It is proved that every continuous, closed and statistically sequence covering image of a metric space is metrizable \cite[Corollary 3.4]{RePr17}.

\begin{question}
Is metrizability preserved by continuous, closed and $\mathcal{I}$-covering mappings?
\end{question}

It is known that a topological space is a sequentially connected space if and only if it is a continuous sequence-covering image of a connected metric space \cite[Theorem 2.3.17]{LsYzq16}.

\begin{question}
How to characterize the spaces as the continuous $\mathcal{I}$-covering images of connected metric spaces?
\end{question}

\end{document}